\documentclass[twoside,10pt,english]{smfart}
\usepackage{url}
\usepackage[english]{babel}
\usepackage[fulloldstylenums,frenchstyle,mathcalasscript,partialup]{kpfonts}
\usepackage{amssymb}
\usepackage[all,2cell,emtex]{xy}

\newtheorem{thm}{Theorem}[section]
\newtheorem{prop}[thm]{Proposition}
\newtheorem{lemma}[thm]{Lemma}
\newtheorem{cor}[thm]{Corollary}

\theoremstyle{remark}
\newtheorem{rem}[thm]{Remark}

\newtheorem{exemple}[thm]{Example}
\newtheorem{paragr}[thm]{}

\theoremstyle{definition}
\newtheorem{definition}[thm]{Definition}

\newcommand{\pref}[1]{{\widehat{ #1 }}}

\newcommand{\To}{\longrightarrow}
\newcommand{\cats}{\Delta}

\newcommand{\simpl}{\pref{\cats}}

\newcommand{\Hom}{\operatorname{\mathrm{Hom}}}
\newcommand{\sHom}{\operatorname{\mathit{Hom}}}

\newcommand{\nerf}{\operatorname{N}}

\newcommand{\op}[1]{{#1}^{\mathit{op}}}

\newcommand{\real}[1]{\operatorname{{\mathit{Real}}^{}_{{#1}}}}
\newcommand{\sing}[1]{\operatorname{{\mathit{Sing}}^{}_{{#1}}}}

\newcommand{\C}{\mathcal{C}}

\newcommand{\smp}[1]{ \varDelta_{#1}}

\newcommand{\e}{\varepsilon}

\newcommand{\bord}{\operatorname{\partial}}

\def\TO#1{\mathrel{\hbox to #1mm{\rightarrowfill}}}
\def\OT#1{\mathrel{\hbox to #1mm{\leftarrowfill}}}

\def\Cube{\mathord{\hskip 1pt\rlap{\vrule height 7.4pt depth 0pt width .8pt
\vrule height 1.2pt depth 0pt width 5.4pt
\vrule height 7.4pt depth 0pt width 1.2pt}
\rlap{\hskip 1.6pt\vrule height 7.4pt depth 0pt width .8pt}
\vrule height 7.4pt depth -6.6pt width 7.4pt\hskip 1pt}}

\def\cube{\mathord{\hskip 1pt\rlap{\vrule height 7.4pt depth 0pt width .4pt
\vrule height .8pt depth 0pt width 6.2pt
\vrule height 7.4pt depth 0pt width .8pt}
\rlap{\hskip 1.pt\vrule height 7.4pt depth 0pt width .4pt}
\vrule height 7.4pt depth -7pt width 7.4pt\hskip 1pt}}

\def\ccornet{\mathord{\hskip 1pt\rlap{\vrule height 7.4pt depth 0pt width .4pt
\vrule height .8pt depth 0pt width .8pt
\vrule height 0pt depth 0pt width 5.4pt 
\vrule height 7.4pt depth 0pt width .8pt}
\rlap{\hskip 1.pt\vrule height 7.4pt depth 0pt width .4pt}
\vrule height 7.4pt depth -7pt width 7.4pt\hskip 1pt}}

\renewcommand{\Cube}{{\Box}}
\renewcommand{\cube}{\Box}
\renewcommand{\ccornet}{\sqcap}

\newcommand{\spref}[1]{s\pref{#1}}
\newcommand{\Eq}{\mathit{Eq}}

\renewcommand{\theequation}{\arabic{equation}}
\numberwithin{equation}{thm}

\title[Univalent universes
for elegant models of homotopy types]{Univalent universes\\
for elegant models of homotopy types}

\author[D.-C. Cisinski]{Denis-Charles Cisinski}
\address{Universit\'e Paul Sabatier\\
Institut de Math\'ematiques de Toulouse\\
118~route de Narbonne\\
31062~Toulouse cedex~9\\France}
\email{denis-charles.cisinski@math.univ-toulouse.fr}
\urladdr{http://www.math.univ-toulouse.fr/~dcisinsk/}

\begin{document}
\frontmatter
\begin{abstract}
We construct a univalent universe
in the sense of Voevodsky in some suitable model categories
for homotopy types (obtained from Grothendieck's theory of
test categories). In practice, this means for instance that, appart from the
homotopy theory of simplicial sets, intensional type theory
with the univalent axiom can be interpreted in the
homotopy theory of cubical sets (with connections or not), or of
Joyal's cellular sets.
\end{abstract}
\maketitle

We recall that any right
proper cofibrantly generated model category structure on a (pre)sheaf topos
whose cofibrations are exactly the monomorphims is a type theoretic model
category in the sense of Shulman \cite[Definition 2.12]{ShInv}. This means that,
up to coherence issues which are solved by
Kapulkin, Lumsdaine, and Voevodsky~\cite{KLV} and by Lumsdaine and Warren~\cite{LuWa},
we can interpret Martin-L\"of intensional type theory in such a model category.
The purpose of these notes is to prove the existence of univalent universes
in suitable model categories for (local systems of) homotopy types, (such as simplicial sets
or cubical sets): presheaves over a local test category in the sense
of Grothendieck which is also elegant in the sense of Bergner and Rezk.
We give two constructions. The first one uses Voevodsky's construction
in the setting of simplicial sets as well as Shulman's extension to simplicial presheaves
over elegant Reedy categories. This has the advantage of giving a rather short
proof, but the disadvantage of giving a non-explicit construction.
The second one consists to develop the theory of
minimal fibrations in the context of presheaves over an Eilenberg-Zilber Reedy category
(which is a slightly more restrictive notion than the one of elegant Reedy category),
following classical approaches (as in \cite{GZ} for instance), and
then to check how robust is Voevodsky's proof. The latter point of view is much
more general and also gives another proof of Shulman's construction of univalent universes,
at least in the case of simplicial presheaves on Eilenberg-Zilber Reedy categories.

The preparation of these notes started after discussions I had with Thierry Coquand about
his joint work with Marc Bezem and Simon Huber on cubical sets \cite{BCH}.
His kind invitation to give a talk at the Institut Henri Poincar\'e
in Paris gave the decisive impulse to turn these into actual mathematics.

\tableofcontents

\section{First construction: reduction to the case of simplicial sets}

We write $\cats$ for the category of simplices (the full subcategory
of the category of small category whose objects are the non-empty
finite totally ordered sets). For an integer $n\geq 0$, we write $\smp n$
for the presheaf represented by the totally ordered set $\{0,\ldots,n\}$.
If $A$ is a small category, we write $\pref{A}$ for the category of presheaves
of sets over $A$, and $\spref A=\pref{A\times\cats}$ for the category of simplicial presheaves
over $A$. We will consider $\pref{A}$ as a full subcategory of $\spref{A}$
(by considering sets as constant presehaves on $\cats$).
Given a cardinal $\kappa$, a morphism of presheaves $X\to Y$ over
a small category $A$ will be said to have $\kappa$-small fibers if, for any object $a$ of $A$
and any section of $Y$ over $a$, the fiber product $a\times_Y X$, seen
as a presheaf over $A/a$, is $\kappa$-accessible.

Let $A$ be an elegant Reedy category (see \cite[Definition 3.2]{BeRe}).
Let us consider the \emph{locally constant model structure}
on the category of simplicial presheaves over $A$, that is the left Bousfield
localization of the injective model structure on the category of simplicial
presheaves by maps of the form $f\times1_{\smp n}:a\times\smp n\to b\times\smp n$ for any
map $f:a\to b$ in $A$ and any integer $n\geq 0$.
The fibrant objects of the locally constant model structure are thus the
injectively fibrant simplicial presheaves $X$ such that, for any map $a\to b$ in $A$, the induced
morphism $X_b\to X_a$ is a simplicial weak homotopy equivalence.
Note that the locally constant model structure is right proper (this
follows from \cite[Theorem 4.4.30 and Corollary 6.4.27]{Ci3}, for instance),
and is thus a type theoretic model category.
We let $\kappa$ an inaccessible cardinal larger than the cardinal of the set of arrows of $A$.

\begin{prop}\label{prop:locallyconstantuniverse}
There exists a univalent universe $\pi:\overline{U}\to U$ in $\spref{A}$
which classifies fibrations with $\kappa$-small fibers
in the locally constant model structure.
\end{prop}

\begin{proof}
Let $p:\overline{V}\to V$ a univalent universe in $\spref A$ with respect to
the injective model structure, classifying fibrations with $\kappa$-small fibers
(this exists by virtue of a result of Shulman; see \cite[Theorem 5.6]{Sh}).
Then, as the $(\infty,1)$-category associated to the
locally constant model structure is an $\infty$-topos (because it corresponds to the
$(\infty,1)$-category of functors from the $\infty$-groupoid of $A$ to the $(\infty,1)$-category
of $\infty$-groupoids), one can apply a general result
of Rezk \cite[Theorem 1.6.6.8]{Lurie}, and get
the existence of a univalent universe up to homotopy with respect to the
locally constant model structure; see \cite[Proposition 6.10]{GeKo}.
In other words, one can find a univalent fibration between fibrant objects $q:\overline{W}\to W$
of the locally constant model structure such that, for any fibration
$f:X\to Y$ of the locally constant model structure, there exists an homotopy pullback square
in $\spref A$ of the following form.
$$\xymatrix{
X\ar[r]\ar[d]_f&\overline{W}\ar[d]^q\\
Y\ar[r]&W
}$$
Let us choose a (strict) pullback square
$$\xymatrix{
\overline{W}\ar[r]^{\overline{w}}\ar[d]_q&\overline{V}\ar[d]^p\\
W\ar[r]^w&V
}$$
and choose a factorization of the map $w$ into a trivial cofibration
$:W\to U$ followed by a fibration $u:U\to V$ (with respect to the
injective model structure). Then $U$ is fibrant for the locally constant model structure.
Moreover, if we put $\overline{U}=U\times_V\overline{V}$, then the map
$\overline{W}\to\overline{U}$ is a trivial cofibration of the injective model structure
(because the latter is right proper) and thus $\overline{U}$ is also fibrant
with respect to the locally constant model structure. The projection
$\pi:\overline{U}\to U$ being a fibration of the injective model structure
between fibrant objects of the locally constant model structure, it is a fibration
for the locally constant model structure. By construction, the projection $\pi$ is
a univalent universe up to homomotopy. It remains to prove that any fibration
of the locally constant model structure can be obtained as a strict pullback of $\pi$.
Let $f:X\to Y$ be a fibration of the locally constant model structure
(with $\kappa$-small fibers).
Then the map $f$ can be obtained as a strict pullback of the universe $p$.
$$\xymatrix{
X\ar[r]^{x}\ar[d]_f&\overline{V}\ar[d]^p\\
Y\ar[r]^y&V
}$$
The classifying map $y:Y\to V$ can be lifted to $U$ up to simplicial homotopy.
Indeed, there exists an homotopy pullback square in the locally constant
model structure of the following form.
$$\xymatrix{
X\ar[r]^{\mu}\ar[d]_f&\overline{U}\ar[d]^\pi\\
Y\ar[r]^\lambda&U
}$$
This square is also an homotopy pullback square in the
injective model structure: the comparison map
$X\to Y\times_U\overline{U}$ is a weak equivalence between
fibrant and cofibrant objects over $Y$ with respect to the model
structure on $\spref A/Y$ induced by the locally constant
model structure, and is thus a simplicial homotopy
equivalence (in particular, a weak equivalence of the
injective model structure), and the map $\pi$ being a fibration
of the injective model structure (which is right proper), this proves our assertion.
Therefore, we obtain the homotopy pullback square below in the
injective model structure.
$$\xymatrix{
X\ar[r]\ar[d]_f&\overline{V}\ar[d]^p\\
Y\ar[r]^{u\lambda}&V
}$$
The fibration $p$ being univalent,
there exists a map $h:\smp 1\times Y\to V$
such that $h|_{\{0\}\times Y}=y$ and $h|_{\{1\}\times Y}=u\lambda$.
As the map $u$ is a fibration of the injective model structure,
the commutative square
$$\xymatrix{
\{1\}\times Y\ar[d]\ar[r]^(.6)\lambda&U\ar[d]^u\\
\smp 1\times Y\ar[r]^(.6)h&V
}$$
admits a lift $l:\smp 1\times Y\to U$. If we define $y':Y\to U$ as $y'=l|_{\{0\}\times Y}$,
then we have $uy'=h|_{\{0\}\times Y}=y$. In other words, we then obtain a strict pullback
square
$$\xymatrix{
X\ar[r]^{x'}\ar[d]_f&\overline{U}\ar[d]^\pi\\
Y\ar[r]^{y'}&U
}$$
which shows that $\pi$ is a univalent universe for the locally constant model structure
in the strict sense.
\end{proof}

\begin{paragr}\label{par:test}
Assume now that $A$ is a local test category in the sense of Grothendieck;
see \cite[Definition 1.5.2 and Theorem 1.5.6]{Ma}.
Then the category $\pref A$ is endowed with a proper combinatorial model structure
with the monomorphisms as cofibrations, while the weak equivalences are the maps
$X\to Y$ which induce a simplicial weak homotopy equivalence $\nerf A/X\to \nerf A/Y$
(where $A/X$ denotes the category of elements of the presheaf $X$, while $\nerf$
is the nerve functor); see \cite[Corollary 4.2.18 and Theorem 4.4.30]{Ci3}.
We will refer to this model structure (which is in particular type theoretic)
as the \emph{Grothendieck model structure}.
Moreover, the inclusion functor $\pref{A}\to\spref A$ is then
a left Quillen equivalence with the locally constant model structure,
with right adjoint the evaluation at zero functor. If, moreover, the category $A$ is elegant,
the universe $\pi:\overline{U}\to U$ obtained in Proposition \ref{prop:locallyconstantuniverse}
induces a fibration between fibrant objects $\pi_0:\overline{U}_0\to U_0$ in $\pref{A}$.
\end{paragr}

\begin{thm}\label{thm:testuniverse}
The fibration $\pi_0:\overline{U}_0\to U_0$ is a univalent universe in $\pref{A}$
which classifies fibrations with $\kappa$-small fibers.
\end{thm}

\begin{proof}
Remark that the category $\cats$ is an example of a (local) test category and that the
Grothendieck model structure on $\simpl$ coincides with the usual model structure
(whose fibrant objects are the Kan complexes).
Let $D:\cats\to\pref{A}$ be a normalized cosimplicial resolution in the sense of \cite[Definition 2.3.12]{Ci3}
(this always exists: see \cite[2.3.13]{Ci3} for a canonical example).
Note that, in particular, $D_0$ is the terminal presheaf.
By virtue of \cite[Proposition 2.3.27, Corollaries 4.4.10 and 6.4.27]{Ci3}, we then have a left Quillen
equivalence from the locally constant model structure to the Grothendieck model structure
$$\real D:\spref A\to\pref A$$
with right adjoint
$$\sing D:\pref A\to\spref A$$
given by the formula
$$\sing D (X)_n=\sHom(D_n,X)$$
(where $\sHom$ denotes the internal Hom of $\pref{A}$ with respect to the
cartesian product). Let $f:X\to Y$ be a fibration of $\pref A$
with $\kappa$-small fibers.
Then, the map $\sing D(f)$ being a fibration
of the locally constant model structure with $\kappa$-small fibers,
there exists a (strict) pullback square of the following form.
$$\xymatrix{
\sing D(X)\ar[r]\ar[d]_{\sing D(f)}&\overline{U}\ar[d]^\pi\\
\sing D(Y)\ar[r]&U
}$$
Evaluating at zero thus gives the (strict) pullback below.
$$\xymatrix{
X\ar[r]\ar[d]_f&\overline{U}_0\ar[d]^{\pi_0}\\
Y\ar[r]&U_0
}$$
As evaluating at zero is a right Quillen equivalence, the map $\pi_0$ is a univalent universe
up to homotopy (because the fibration $\pi$ has this property),
and what precedes thus proves that it is a univalent universe in the strict sense.
\end{proof}

\begin{cor}\label{cor:elegantunivalence}
For any elegant local test category $A$, the Grothendieck model structure
on the category of presheaves of sets $\pref A$
supports a model of intensional type theory with dependent sums and products,
identity types, and as many univalent universes as there are inaccessible cardinals
greater than the set of arrows of $A$.
\end{cor}

The most well known example of elegant local test categories is provided by the
category of simplices $\cats$, the Grothendieck model structure
on $\simpl$ being then the standard model structure of simplicial sets.
The preceding corollary is in this case Voevodsky's theorem that univalence
holds in simplicial sets; see \cite{KLV}. As the proof of this corollary
relies on Voevodsky's results, the interesting examples for us are of course
the other elegant (local) test categories. We give a few examples below.

\begin{exemple}\label{example:cube}
The category of cubical sets supports a Grothendieck model structure.
Indeed, cubical sets are presheaves on the category $\Cube$ which is
defined as follows. Let $\overline{\Cube}$ be the full subcategory of the category
of sets whose objects are the sets $\cube_n=\{0,1\}^n$, $n\geq 0$.
This is a (strict) symmetric monoidal category (with cartesian product
as tensor product). The category $\Cube$ is defined as the smallest
monoidal subcategory of $\overline{\Cube}$ generated by the maps
\begin{align*}
\delta^0&:\cube_0\to\cube_1\\
\delta^1&:\cube_0\to\cube_1\\
\sigma &:\cube_1\to\cube_0
\end{align*}
where $\delta^e$ is the constant map with value $e$.
For $n\geq 1$, $1\leq i\leq n$ and $e=0,1$, one defines a map
$$\delta^{i,e}_n=1_{\cube_{i-1}}\otimes\delta^e\otimes1_{\cube_{n-i}}:\cube_{n-1}\to\cube_n\, ,$$
and for $n\geq 0$ and $1\leq i\leq n$, a map
$$\sigma^i_n=1_{\cube_{i-1}}\otimes\sigma\otimes1_{\cube_{n+1-i}}:\cube_{n+1}\to\cube_n\, .$$
The category $\Cube$ is an elegant test category;
see \cite[Cor. 8.4.13, Prop. 8.1.24 and 8.4.17]{Ci3}.

Moreover, the Grothendieck model structure on $\pref\Cube$ admits the
following explicit description. We will consider the Yoneda embedding
as an inclusion, and thus write $\cube_n$ for the
presheaf represented by $\cube_n$ for each $n\geq 0$. The boundary of $\cube_n$
is defined as the union of the images of the maps $\delta^{i,e}_n$ for
$1\leq i\leq n$ and $e=0,1$, and will be denoted by $\bord\cube_n$.
We also define, for $1\leq i\leq n$ and $e=0,1$, the presheaf $\ccornet^{i,e}_n$
as the union in $\cube_n$ of the images of the maps $\delta^{j,\varepsilon}_n$ for
$(j,\varepsilon)\neq(i,e)$. Then the generating cofibrations (trivial cofibrations)
of the Grothendieck model structure on $\pref{\Cube}$ are the inclusions
$$\bord\cube_n\to\cube_n \, , \ n\geq 0 \quad
\text{($\ccornet^{i,e}_n\to\cube_n \, , \ n\geq 1\, , \ 1\leq i\leq n \, , \ e=0,1$, respectively.)}$$
In other words, the fibrations are precisely the cubical Kan fibrations; see \cite[Theorem 8.4.38]{Ci3}.
\end{exemple}

\begin{exemple}\label{example:cubical}
Cubical sets with connections also give an example of a category of presheaves
over an elegant test category. The category $\Cube^c$ is defined as the
smallest monoidal subcategory of $\overline{\Cube}$ generated by the
maps $\delta^0$, $\delta^1$ and $\sigma$ as above, as well as by the map
$$\gamma:\cube_2=\cube_1\otimes\cube_1\to\cube_1$$
given by $\gamma(x,y)=\mathrm{sup}\{x,y\}$.
Cubical sets with connections are presheaves of sets on the category $\Cube^c$.
The category $\Cube^c$ is an elegant test category for the same reasons as for $\Cube$
(applying \cite[Prop. 8.1.24 and 8.4.12]{Ci3} for instance). But in fact, as was proved
by Maltsiniotis \cite[Prop. 3.3]{Ma7}, it has a better property:
it is a strict test category, which means that the weak equivalences
of the Grothendieck model category structure on the category of cubical sets
with connections are closed under finite products (while this property is known to fail
for cubical sets without connections).
We define the inclusions
$$\bord\cube_n\to\cube_n\quad\text{and}\quad \ccornet^{i,e}_n\to\cube_n$$
in the same way as for cubical sets above. Note that the category
of cubical sets with connections has a natural (non symmetric) closed monoidal
structure induced by the monoidal structure on $\Cube^c$ (using
Day convolution). The following theorem is a rather direct consequence
of its analog for cubical sets, but we state and prove it explicitely
for the convenience of the reader and for future reference.
\end{exemple}

\begin{thm}
The Grothendieck model category structure on the category of
cubical sets with connections is a proper cofibrantly generated
monoidal model category with generating cofibrations
$$\bord \cube_n \To \cube_n \quad , \quad n \geq 0 \ , $$
and generating trivial cofibrations
$$\ccornet^{i,e}_n \To \cube_n \quad , \quad
n \geq 1 \ , \ 1 \leq i \leq n \ , \ e = 0,1 \ . $$
Furthermore, the class of weak equivalences is closed under
finite products.
\end{thm}

\begin{proof}
As the category $\Cube^c$ is elegant,
it is clear that the boundary inclusions $\bord \cube_n \To \cube_n$
generate the class of monomorphisms.
We have an obvious inclusion functor $u:\Cube\to\Cube^c$
which provides an adjunction
$$u_!:\pref{\Cube}\rightleftarrows{\pref{\Cube}}^c:u^*$$
where $u_!$ is the left Kan extension of $u$.
The functor $u_!$ is symmetric monoidal and sends
$\bord \cube_n$ as well as $\ccornet^{i,\e}_n$
to their versions with connections (it is sufficient to prove
this for $\bord\cube_n$, which follows from
\cite[Lemma 8.4.21]{Ci3}). Therefore,
the functor $u_!$ preserves monomorphisms,
and, as any map between representable cubical sets with connections
is a weak equivalence, using the last assertion of \cite[8.4.27]{Ci3},
we see immediately that the functor $u_!$ is a left Quillen functor.
In particular, the inclusions $\ccornet^{i,\e}_n \To \cube_n$
are trivial cofibrations of cubical sets with connections.
The theorem now follows straight away from \cite[Lemma 8.4.37]{Ci3}
which allows to apply \cite[Lemma 8.2.17 and Theorem 8.2.18]{Ci3}
(the last assertion is due to the fact, that $\Cube^c$
is a strict test category).
\end{proof}

\begin{exemple}\label{example:joyal}
For $1\leq n\leq \omega$, one can consider Joyal's category $\Theta_n$,
which is to strict $n$-categories what $\cats$ is to categories (in particular,
$\Theta_1=\cats$). The category $\Theta_n$ can be thought of as the
full subcategory of the category of strict $n$-categories
whose objects are the strict $n$-categories freely generated
on finite pasting schemes of dimension $\leq n$
(Joyal also gave a description of $\op{\Theta}_n$ as the category of $n$-disks).
We know from \cite[Cor. 4.5]{BeRe} that $\Theta_n$ is elegant, and from
\cite[Examples 5.8 and 5.12]{CisMal} that $\Theta_n$ is a strict test category.
Therefore, the category of presheaves of sets on $\Theta_n$ carries a Grothendieck
model category structure in which the class of weak equivalences is closed under
finite products.
\end{exemple}

\section{Minimal fibrations}

\begin{paragr}\label{paragr:EZcat}
In this section, we fix once and for all an \emph{Eilenberg-Zilber category} $A$.
This means that $A$ is a Reedy category with the following properties.
\begin{itemize}
\item[(EZ1)] Any map in $A_-$ has a section in $A$.
\item[(EZ2)] If two maps in $A_-$ have the same set of sections, then they
are equal.
\end{itemize}
Such a Reedy category is elegant; see \cite[Proposition 4.2]{BeRe}.
We also consider given
a model category structure on $\pref A$, whose cofibrations
precisely are the monomorphisms.
Given a representable presheaf $a$, we denote by $\bord a\to a$ the
boundary inclusion (which means that $\bord a$ is the maximal proper subobject of $a$).
These inclusions form a generating set for the class of cofibrations.
We also choose an interval $I$ such that the projection $I\times X\to X$
is a weak equivalence for any presheaf $X$ on $A$ (e.g. we can take for $I$
the subobject classifier of the topos $\pref A$). We write $\bord I=\{0\}\amalg\{1\}\subset I$
for the inclusion of the two end-points of $I$.
\end{paragr}

\begin{paragr}
Let $h:I\times X\to Y$ be an homotopy. For $e=0,1$, we write $h_e$ for the composite
$$X=\{e\}\times X\to I\times X\xrightarrow{ \ h \ } Y\, .$$
Given a subobject $S\subset X$, we say that $h$ is \emph{constant on $S$} if the
restriction $h|_{I\times S}$ factors through the second projection $I\times S\to S$.

Given a presheaf $X$, a \emph{section of $X$} is a map $x:a\to X$ with $a$
a representable presheaf. The \emph{boundary} of such a section $x$ is the map 
$$\bord x:\bord a\to a\xrightarrow{ \ x \ }X\, .$$
\end{paragr}

\begin{definition}
Let $X$ be an object of $\pref{A}$.

Two sections $x,y:a\to X$ are \emph{$\bord$-equivalent} if the following
conditions are satisfied.
\begin{itemize}
\item[(i)] These have the same boundaries: $\bord x=\bord y$.
\item[(ii)] There exists an homotopy $h:I\times a\to X$ which is constant
on $\bord a$, and such that $h_0=x$ and $h_1=y$.
\end{itemize}
We write $x\simeq y$ whenever $x$ and $y$ are $\bord$-equivalent.

A \emph{minimal complex} is a fibrant object $S$ such that, for any two
sections $x,y:a\to S$, if $x$ and $y$ are $\bord$-equivalent, then $x=y$.

A \emph{minimal model} of $X$ is a trivial cofibration $S\to X$ with $S$
a minimal complex.
\end{definition}

\begin{prop}\label{prop:bordequiv-equivrel}
Let $X$ be a fibrant object. The $\bord$-equivalence relation is
an equivalence relation.
\end{prop}

\begin{proof}
This can be proved directly (exercise). Here is a fancy
argument. The interval $I$ defines an enrichment of
the category $\pref A$ over the category of cubical sets:
given two presheaves $E$ and $F$ over $A$, the cubical set
$\mathrm{Map}(E,F)$ is defined by
$$\mathrm{Map}(E,F)_n=\Hom_{\pref A}(E\times I^n,F)$$
for $n\geq 0$, with $I^n$ the cartesian product of $n$ copies of $I$.
The functor $\mathrm{Map}(E,-)$ is a right Quillen functor
to the Grothendieck model category structure on the category
of cubical sets (see Example \ref{example:cubical}). Therefore,
given a section $x:a\to X$, we can form the following
pullback square.
$$\xymatrix{
X(\bord x)\ar[r]\ar[d]&\mathrm{Map}(a,X)\ar[d]\\
e\ar[r]^(.4){\bord x}&\mathrm{Map}(\bord a,X)
}$$
(where $e$ denotes the terminal cubical set).
The section $x$ can be seen as a global section of the cubical Kan complex $X(\bord x)$.
The relation of cubical homotopy is an equivalence relation on the
set of points of $X(\bord x)$.
Finally, note that, if $y:a\to X$ is another section such that $\bord x=\bord y$,
then $X(\bord x)=X(\bord y)$, from which we deduce right away this proposition.
\end{proof}

\begin{prop}\label{prop:weakbordequiv}
Let $\varepsilon\in\{0,1\}$, and consider a fibrant object $X$, together
with two maps $h,k:I\times a\to X$, with $a$ representable,
such that the restrictions of $k$ and $h$ coincide on
$I\times\bord a\cup\{1-\varepsilon\}\times a$.
Then the sections $h_\varepsilon$ and $k_\varepsilon$
are $\bord$-equivalent.
\end{prop}

\begin{proof}
Put $x=h_\varepsilon$, $y=k_\varepsilon$, and $z=h_{1-\varepsilon}=k_{1-\varepsilon}$.
We will use the cubical mapping spaces as in the proof of
the preceding proposition. We can think of $h$ and $k$ as two
fillings in the commutative square
$$\xymatrix{
\{1-\varepsilon\}\ar[r]^(.42)z\ar[d]&\mathrm{Map}(a,X)\ar[d]\\
\cube_1\ar@{..>}[ur]^{\text{$h$ or $k$}}\ar[r]^(.4)\xi&\mathrm{Map}(\bord a,X)
}$$
in which $\xi$ correspond to the restriction of $h$ (or $k$) to $I\times\bord a$.
As maps in the homotopy category of pointed cubical sets over $\mathrm{Map}(\bord a,X)$,
we must have $h=k$. Therefore, there exists a map
$H:\cube_2\to\mathrm{Map}(a,X)$ with the following properties.
The restriction of $H$ to $\{0\}\otimes\cube_1=\cube_1$
(resp. to $\{1\}\otimes\cube_1=\cube_1$) is $h$ (resp. $k$), the restriction
to $\cube_1\otimes\{1-\varepsilon\}$ is constant with value $z$, and the following
diagram commutes.
$$\xymatrix{
\cube_2\ar[r]^(.4)H\ar[d]_{1\otimes\sigma}&\mathrm{Map}(a,X)\ar[d]\\
\cube_1\ar[r]^(.4)\xi&\mathrm{Map}(\bord a,X)
}$$
The restriction of $H$ to $\cube\otimes\{\varepsilon\}$ defines
a map $l:I\times a\to X$ which is constant on $\bord a$ and such that
$l_0=x$ and $l_1=y$.
\end{proof}

\begin{lemma}\label{lemma:degbordequivequal}
Let $X$ be a fibrant object, and $x_0,x_1:a\to X$ two degenerate sections.
If $x_0$ and $x_1$ are $\bord$-equivalent, then they are equal. 
\end{lemma}

\begin{proof}
For $\varepsilon=0,1$, there is a unique couple $(p_\varepsilon,y_\varepsilon)$, where
$p_\varepsilon:a\to b_\varepsilon$ is a split epimorphism in $A$
and $y_\varepsilon:b_\varepsilon\to X$ is a non-degenerate section of $X$ such that
$x_\varepsilon=y_\varepsilon p_\varepsilon$.
Let us choose a section $s_\varepsilon$ of $p_\varepsilon$.
As $x_0$ and $x_1$ are degenerate and since $\bord x_0=\bord x_1$,
we have $x_0s_0=x_1s_0$ and $x_0s_1=x_1s_1$.
On the other hand, we have $y_\varepsilon=x_\varepsilon s_\varepsilon$.
We thus have the equalities $y_0=y_1p_1s_0$ and $y_1=y_0p_0s_1$.
These imply that the maps $p_\varepsilon s_{1-\varepsilon}:b_{1-\varepsilon}\to b_\varepsilon$
are in $A_{+}$ and that $b_0$ and $b_1$ have the same dimension.
This means that $p_\varepsilon s_{1-\varepsilon}$
is the identity for $\varepsilon=0,1$. In other words, we have $b_0=b_1$ and $y_0=y_1$,
and we also have proven that $p_0$ and $p_1$ have the same sections, whence are equal.
\end{proof}

\begin{thm}\label{thm:existenceminimal}
Any fibrant object has a minimal model.
\end{thm}

\begin{proof}
Let $X$ be a fibrant object. We choose a representant of each $\bord$-equivalence
class. A section of $X$ which is a chosen representative of its $\bord$-equivalence
class will be called \emph{selected}. By virtue of Lemma \ref{lemma:degbordequivequal},
we may assume that any degenerate section $X$ is selected.
Let $E$ be the set of subobjects $S$ of $X$ such that any section of $S$ is selected
Then $E$ is not empty: the image of any selected section of the $0$-skeleton of $X$
is an element of $E$.
By Zorn's lemma, we can choose a maximal element $S$ of $E$
(with respect to inclusion). Remark that any selected section $x:a\to X$
whose boundary $\bord x$ factors through $S$ must belong to $S$.
Indeed, if $x$ is degenerate, then it factors through $\bord a$ hence through $S$.
Otherwise, let us consider $S'=S\cup\mathrm{Im}(x)$. A non-degenerate section
of $S'$ must either factor through $S$ or be precisely equal to $x$.
In any case, such a section must be selected, and the maximality of $S$ implies
that $S=S'$.

We will prove that $S$ is a retract of $X$ and that the inclusion
$S\to X$ is an $I$-homotopy equivalence. This will prove that $S$ is fibrant
and thus a minimal model of $X$. Let us write $i:S\to X$ for the inclusion map.
Consider triples $(T,h,p)$, where $T$ is a subobject of $X$ which contains $S$,
$p:T\to S$ is a retraction (i.e. the restriction of $p$ to $S$ is the identity), and
$h:I\times T\to X$ is a map which is constant on $S$, and such that $h_0$ is the inclusion map $T\to X$, while
$h_1=ip$. Such triples are ordered in the obvious way: $(T,h,p)\leq(T',h',p')$
if $T\subset T'$, with $h'|_{I\times T}=h$ and $p'_T=p$. By Zorn's lemma, we can choose
a maximal triple $(T,h,p)$. To finish the proof, it is sufficient to prove that $T=X$.
In other words, it is sufficient to prove that any non-degenerate section of $X$
belongs to $T$. Let $x:a\to X$ be a non-degenerate section which does not belong
to $T$. Assume that the dimension of $a$ is minimal for this property.
Then $\bord x$ must factor through $T$, so that, if we define $T'$ to be
the union of $T$ and of the image of $x$ in $X$, then we have a bicartesian square
of the following form.
$$\xymatrix{
\bord a\ar[r]^{\bord x}\ar[d]& T\ar[d]\\
a\ar[r]^x&T'
}$$
We have a commutative square
$$\xymatrix{
\{0\}\times\bord a\ar[r]\ar[d]&I\times\bord a\ar[d]^{h(1\times \bord x)}\\
\{0\}\times a\ar[r]^{x}& X
}$$
If we put $u=(h(1\times \bord x),x):I\times\bord a\cup\{0\}\times a\to X$,
we can choose a map $H:I\times a\to X$ such that $H_0=x$, while
$H_{|_{I\times\bord a}}=h(1\times \bord x)$. If we write $y_0=H_1$, as $h_1$
factors through $S$, we see that the
boundary $\bord y_0$ must factor through $S$. Let $y$ be the selected section
$\bord$-equivalent to $y_0$. We choose an homotopy $K:I\times a\to X$
which is constant on $\bord a$ and such that $K_0=y_0$ and $K_1=y$.
By an easy path lifting argument (composing the homotopies $H$ and $K$),
we see that we may choose $H$ such that $y=y_0$.
Note that, as $y$ is selected with boundary in $S$, we must have $y$ in $S$.
We obtain the commutative diagram
$$\xymatrix{
I\times\bord a\ar[r]^{1_I\times\bord x}\ar[d]&I\times T\ar[d]^h\\
I\times a\ar[r]^H&X
}$$
so that, identifying $I\times T'$ with $I\times a\amalg_{I\times\bord a}I\times T$,
we define $h'=(H,h):I\times T'\to X$. Similarly, the
commutative diagram
$$\xymatrix{
\bord a\ar[r]^{\bord x}\ar[d]&T\ar[d]^p\\
a\ar[r]^y&S
}$$
defines a map $p'=(y,p):T'=a\amalg_{\bord a}T\to X$.
It is clear that the triple $(T',h',p')$ extends $(T,h,p)$, which leads
to a contradiction.
\end{proof}

\begin{prop}\label{prop:retractfibtriv}
Let $X$ be a fibrant object and $i:S\to X$ a minimal resolution of $X$.
Consider a map $r:X\to S$ such that $ri=1_S$
(such a map always exists because $i$ is a trivial cofibration
with fibrant domain). Then the map $r$ is a trivial fibration.
\end{prop}

\begin{proof}
There exists a map
$h:I\times X\to X$ which is constant on $S$ and such that $h_0=ir$
and $h_1=1_X$: we can see $i$ as a trivial cofibation between
cofibrant and fibrant objects in the model category of objects
under $S$, and $r$ is then an inverse up to homotopy in this
relative situation. Consider the commutative diagram below.
$$\xymatrix{
\bord a\ar[r]^u\ar[d]&X\ar[d]^r\\
a\ar[r]^v&S
}$$
We want to prove the existence of a map $w:a\to X$ such that $w_{|_{\bord a}}=u$
and $rw=v$. As $X$ is fibrant, there exists a map $k:I\times a\to X$
whose retriction to $I\times\bord a$ is $h(1_I\times u)$, while $k_0=iv$.
Let us put $w=k_1$. Then
$$\bord w=w_{|_{\bord a}}=(h(1_I\times u))_1=h_1u=u\, .$$
It is thus sufficient to prove that $v=rw$.
But $k$ and $h(1_I\times w)$ coincide
on $I\times\bord a\cup\{1\}\times a$ and thus, by virtue of Proposition \ref{prop:weakbordequiv},
we must have $k_0\sim(h(1_I\times w))_0$.
In other words, we have $iv\sim irw$.
As $ri=1_S$, this implies that $v\sim rw$, and, by minimality of $S$, that $v=rw$.
\end{proof}

\begin{lemma}\label{lemma:fhomotopicidentityiso}
Let $X$ be a minimal complex and $f:X\to X$ a map
which is $I$-homotopic to the identity. Then $f$ is
an isomorphism.
\end{lemma}

\begin{proof}
Let us choose once and for all a map $h:I\times X\to X$
such that $h_0=1_X$ and $h_1=f$.
We will prove that the map $f_a:X_a\to X_a$
is bijective by induction on the dimension $d$ of $a$.
If $a$ is of dimension $<0$, there is nothing to prove
because there is no such $a$.
Assume that the map $f_b:X_b\to X_b$ is bijective
for any object $b$ of dimension $<d$.
Consider two sections $x,y:a\to X$ such that $f(x)=f(y)$.
Then, as $f$ is injective in dimension lesser than $d$, the
equations
$$f\bord x=\bord f(x)=\bord f(y)=f\bord y$$
imply that $\bord x=\bord y$.
On the other hand, we can apply Proposition \ref{prop:weakbordequiv}
to the maps $h(1_I\times x)$ and $h(1_I\times y)$ for $\varepsilon=0$,
and we deduce that $x\sim y$. As $X$ is minimal, this proves that $x=y$.
It remains to prove the surjectivity. Let $y:a\to X$ be a section.
For any map $\sigma:b\to a$ in $A$ such that $b$ is of degree lesser than $d$,
there is a unique section $x_\sigma:b\to X$ such that $f(x_\sigma)=\sigma^*(y)=y\sigma$.
This implies that there is a unique map $z:\bord a\to X$ such that $fz=\bord y$.
The map
$I\times\bord a\xrightarrow{ 1_I\times z}I\times a\xrightarrow{ \ h \ }X$,
together with the map
$\{1\}\times a=a\xrightarrow{ \ y \ }X$,
define a map $\varphi=(h(1_I\times z),x)$, and we can choose a filling $k$
in the diagram below.
$$\xymatrix{
I\times\bord a\cup\{1\}\times a\ar[r]^(.6)\varphi\ar[d]&X\\
I\times a\ar@{..>}[ur]^k&
}$$
Let us put $x=k_0$. Then $\bord x=z$, and thus $\bord f(x)=\bord y$.
Applying Proposition \ref{prop:weakbordequiv} to the
maps $k$ and $h(1_I\times x)$ for $\varepsilon=1$, we conclude that $f(x)\sim y$.
The object $X$ being a minimal complex, this proves that $f(x)=y$.
\end{proof}

\begin{prop}\label{prop:weqminimaliso}
Let $X$ and $Y$ be two minimal complexes. Then any
weak equivalence $f:X\to Y$ is an isomorphism of presheaves.
\end{prop}

\begin{proof}
If $f:X\to Y$ is a weak equivalence, as both $x$ and $Y$ are cofibrant
and fibrant, there exists $g:Y\to X$ such that $fg$ and $gf$ are
homotopic to the identify of $Y$ and of $X$, respectively.
By virtue of the preceding lemma, the maps $gf$ and $fg$ must
be isomorphisms, which imply right away that $f$ is an isomorphism.
\end{proof}

\begin{thm}\label{thm:caractminimal}
Let $X$ be a fibrant object of $\pref{A}$. The following conditions
are equivalent.
\begin{itemize}
\item[(i)] The object $X$ is a minimal complex.
\item[(ii)] Any trivial fibration of the form $X\to S$ is an isomorphim.
\item[(iii)] Any trivial cofibration of the form $S\to X$, with $S$ fibrant,
is an isomorphism.
\item[(iv)] Any weak equivalence $X\to S$, with $S$ a minimal complex,
is an isomorphism.
\item[(v)] Any weak equivalence $S\to X$, with $S$ a minimal complex,
is an isomorphism.
\end{itemize}
\end{thm}

\begin{proof}
It follows immediately from Proposition \ref{prop:weqminimaliso}
that condition (i) is equivalent to condition (iv)
as well as to condition (v).
Therefore, condition (v) implies condition (iii):
if $i:S\to X$ is a trivial cofibration with $S$ fibrant and $X$ minimal,
then $S$ must be minimal as well, so that $i$ has to be an isomorphism.
Let us prove that condition (iii) implies condition (ii):
any trivial fibration $p:X\to S$ admits a section $i:S\to X$ which has to be a trivial
cofibration with fibrant domain, and thus an isomorphism.
It is now sufficient to prove that condition (ii) implies condition (i).
By virtue of Theorem \ref{thm:existenceminimal}, there exists a
minimal model of $X$, namely a trivial cofibration $S\to X$ with
$S$ a minimal complex. This cofibration has a retraction which,
by virtue of Proposition \ref{prop:retractfibtriv},
is a trivial fibration. Condition (ii) implies that $S$ is isomorphic to $X$, and
thus that $X$ is minimal as well.
\end{proof}

\begin{definition}
A fibration $p:X\to Y$ in $\pref{A}$ is \emph{minimal} if it is
a minimal complex as an object of $\pref{A}/Y=\pref{A/Y}$ for the induced
model category structure (whose, weak equivalences, fibrations and cofibrations
are the maps which have the corresponding property in $\pref{A}$, by
forgetting the base).
\end{definition}

\begin{prop}\label{prop:stabilityminfib}
The class of minimal fibrations is stable by pullback.
\end{prop}

\begin{proof}
Consider a pullback square
$$\xymatrix{
X\ar[r]^u\ar[d]_p&X'\ar[d]^{p'}\\
Y\ar[r]^v&Y'
}$$
in which $p'$ is a minimal fibration.
Let $x,y:a\to X$ two global sections which are $\bord$-equivalent
over $Y$ (i.e. $\bord$-equivalent in $X$, seen as a fibrant object of $\pref{A/Y}$).
Then $u(x)$ and $u(y)$ are $\bord$-equivalent in $X'$ over $Y'$, and
thus $u(x)=u(y)$. As $p(x)=p(y)$, this means that $x=y$.
In other words, $p$ is a minimal fibration.
%
\end{proof}

Everything we proved so far about minimal complexes has its counterpart
in the language of minimal fibrations. Let us mention the properties
that we will use later.

\begin{thm}\label{thm:minimalfibrations}
For any fibration $p:X\to Y$, there exists a trivial fibration
$r:X\to S$ and a minimal fibration $q:S\to Y$ such that $p=qr$.
\end{thm}

\begin{proof}
By virtue of Theorem \ref{thm:existenceminimal} applied to $p$,
seen as a fibrant presheaf over $A/Y$, there exists a trivial
cofibration $i:S\to X$ such that $q=p_{|_S}:S\to Y$ is a minimal
fibration. As both $X$ and $S$ are fibrant (as presheaves over $A/Y$),
the embedding $i$ is a strong deformation retract, so that, by
virtue of Proposition \ref{prop:retractfibtriv} (applied again
in the context of presheaves over $A/Y$), there exists
a trivial fibration $r:X\to S$ such that $ri=1_S$, and
such that $qr=p$.
\end{proof}

\begin{rem}
In the factorisation $p=qr$ given by the preceding theorem,
$q$ is necessarily a retract of $p$. Therefore, if $p$ belongs
to a class of maps which is stable under retracts, the
minimal fibration must have the same property.
Similarly, as $r$ is a trivial fibration, if $p$ belongs to a class
which is defined up to weak equivalences, then so does $q$.
This means that this theorem can be used to study classes
of fibrations which are more general than classes of fibrations
of model category structures.
\end{rem}

\begin{prop}\label{prop:weqminfibrationiso}
For any minimal fibrations $p:X\to Y$ and $p':X'\to Y$,
any weak equivalence $f:X\to X'$ such that $p'f=p$ is an isomorphism.
\end{prop}

\begin{proof}
This is a reformulation of Proposition \ref{prop:weqminimaliso}
in the context of presheaves over $A/Y$.
\end{proof}

\begin{lemma}\label{lemma:extensiontrivfib}
For any cofibration $v:Y\to Y'$ and any trivial fibration $p:X\to Y$, there exists
a trivial fibration $p':X'\to Y'$ and a pullback square of the following form.
$$\xymatrix{
X\ar[r]^u\ar[d]_p&X'\ar[d]^{p'}\\
Y\ar[r]^v&Y'
}$$
\end{lemma}

\begin{proof}
We use Joyal's trick. The pullback functor $v^*:\pref{A}/Y'\to\pref{A}/Y$
has a left adjoint $v_!$ and a right adjoint $v_*$.
We see right away that $v^*v_!$ is isomorphic to the identity
(i.e. that $v_!$ is fully faithful), so that, by transposition,
$v^*v_*$ is isomorphic to the identity as well.
Moreover, the functor $v_*$ preserves trivial fibrations because
its left adjoint $v^*$ preserves monomorphisms. We define the trivial
fibration $p'$ as $v_*(p:X\to Y)$.
\end{proof}

\begin{prop}\label{prop:extensionweakequivcof}
Consider a commutative diagram of the form
$$\xymatrix{
X_0\ar[r]^w\ar[dr]_{p_0}&X_1\ar[r]^{i_1}\ar[d]^{p_1}&X'_1\ar[d]^{p'_1}\\
&Y\ar[r]^j&Y'
}$$
in which $p_0$, $p_1$ and $p'_1$ are fibrations, $w$ is a weak
equivalence, $j$ is a cofibration, and the square is cartesian. Then there exists
a cartesian square
$$\xymatrix{
X_0\ar[r]^{i_0}\ar[d]^{p_0}&X'_0\ar[d]^{p'_0}\\
Y\ar[r]^j&Y'
}$$
in which $p_0$ is a fibration, as well as a weak equivalence
$w':X'_0\to X'_1$ such that $p'_1w=p'_0$ and $i_1w=w'i_0$.
\end{prop}

\begin{proof}
By virtue of Theorem \ref{thm:minimalfibrations}, we can
choose a trivial fibration $r'_1:X'_1\to S'$ and
a minimal fibration $q':S\to Y'$ such that $p'_1=q'r'_1$.
Let us write $S=Y\times_{Y'}S'$, and $k:S\to S'$
for the second projection. The canonical map
$r_1:X_1\to S$ is a trivial fibration (being the pullback
of such a thing), and the projection $q:S\to Y$ is a minimal
fibration by Proposition \ref{prop:stabilityminfib}. We have
thus a factorisation $p_1=qr_1$.
Moreover, the map $r_0=r_1w$ is a trivial fibration.
To see this, let us choose a minimal model $u: T\to X_0$.
Then the map $r_1wu$ is a weak equivalence between
minimal fibrations and is thus an isomorphism by
Proposition \ref{prop:weqminfibrationiso}.
This means that $r_0$ is isomorphic to a retraction of
the map $u$, and is therefore a trivial fibration
by Proposition \ref{prop:retractfibtriv}.
The diagram we started from now has the following form.
$$\xymatrix{
X_0\ar[r]^w\ar[dr]_{r_0}&X_1\ar[r]^{i_1}\ar[d]^{r_1}&X'_1\ar[d]^{r'_1}\\
&S\ar[r]^k\ar[d]^q&S'\ar[d]^{q'}\\
&Y\ar[r]^j&Y'
}$$
Moreover, both squares are cartesian.
This means that we can replace $j$ by $k$. In other words,
without loss of generality, it is sufficient to prove the proposition
in the case where $p_0$, $p_1$ and $p'_1$ are trivial fibrations.
Under these additional assumptions, we obtain a cartesian square
$$\xymatrix{
X_0\ar[r]^{i_0}\ar[d]^{p_0}&X'_0\ar[d]^{p'_0}\\
Y\ar[r]^j&Y'
}$$
in which $p'_0$ is a trivial fibration by Lemma \ref{lemma:extensiontrivfib}.
The lifting problem
$$\xymatrix{
X_0\ar[d]^{i_0}\ar[r]^w&X_1\ar[r]^{i_1}&X'_1\ar[d]^{p'_1}\\
X'_0\ar[rr]^{p'_0}\ar@{..>}[urr]^{w'}&&Y'
}$$
has a solution because $i_0$ is a cofibration and $p'_1$
a trivial fibration. Moreover, any lift $w'$ must be
a weak equivalence because both $p'_0$ and $p'_1$
are trivial fibrations.
\end{proof}

\begin{rem}
One can prove the preceding proposition in a much
greater generality, without using the theory of minimal fibrations:
the proof of \cite[Theorem 3.4.1]{KLV} can be carried out in any topos
endowed with a model category structure whose cofibrations precisely are
the monomorphisms.
\end{rem}

\begin{lemma}\label{lemma:extensionminimalfib}
Assume that the model category structure on $\pref A$ is right proper.
For any trivial cofibration $v:Y\to Y'$ and any minimal fibration $p:X\to Y$, there exists
a minimal fibration $p':X'\to Y'$ and a pullback square of the following form.
$$\xymatrix{
X\ar[r]^u\ar[d]_p&X'\ar[d]^{p'}\\
Y\ar[r]^v&Y'
}$$
\end{lemma}

\begin{proof}
Let us factor the map $vp$ as a trivial cofibration $u':X\to X''$
followed by a fibration $p'':X''\to Y'$. By virtue of Theorem \ref{thm:minimalfibrations},
we can factor $p''$ into a trivial fibration $q:X''\to X'$ followed by
a minimal fibration $p':X'\to Y'$. We thus get a commutative square
$$\xymatrix{
X\ar[r]^u\ar[d]_p&X'\ar[d]^{p'}\\
Y\ar[r]^v&Y'
}$$
with $u=qu'$. The projection $Y\times_{Y'}X'\to Y$ is a minimal fibration (Proposition \ref{prop:stabilityminfib}).
On the other hand, the model category structure being right proper, the
comparison map $X\to Y\times_{Y'}X'$ is a weak equivalence over $Y$.
Therefore, Proposition \ref{prop:weqminfibrationiso} implies that this comparison
map is an isomorphism, and thus that this commutative square is cartesian.
\end{proof}

\begin{prop}\label{prop:extensionfib}
Assume that the model category structure on $\pref A$ is right proper.
For any trivial cofibration $v:Y\to Y'$ and any fibration $p:X\to Y$, there exists
a fibration $p':X'\to Y'$ and a pullback square of the following form.
$$\xymatrix{
X\ar[r]^u\ar[d]_p&X'\ar[d]^{p'}\\
Y\ar[r]^v&Y'
}$$
\end{prop}

\begin{proof}
By virtue of Theorem \ref{thm:minimalfibrations},
there exists a factorisation of $p$ as $p=qr$ with $r$
a trivial fibration and $q$ a minimal fibration.
We can extend $q$ and then $r$, using Lemmata
\ref{lemma:extensionminimalfib} and \ref{lemma:extensiontrivfib}
successively.
\end{proof}

\section{Second construction: extension of Voevodsky's proof}

\begin{paragr}
Let $A$ be a small category.
A class of presheaves $\C$ on $A$ is \emph{saturated by monomorphisms}
is it satisfies the following properties.
\begin{itemize}
\item[(a)] The empty presheaf is in $\C$.
\item[(b)] For any pushout square
$$\xymatrix{
X\ar[r]^u\ar[d]_i&X'\ar[d]^{i'}\\
Y\ar[r]^v&Y'
}$$
in which $X$, $X'$ and $Y$ are in $\C$ and $i$ is a monomorphism, then $Y'$ is in $\C$.
\item[(c)] For any well ordered set $\alpha$ and any functor $X:\alpha\to\pref{A}$
such that the natural map $X_i\to X_j$ is a monomorphism for any $i<j$ in $\alpha$,
if $X_i$ is in $\C$ for any $i\in\alpha$, then $\varinjlim_{i\in\alpha}X_i$ is in $\C$.
\item[(d)] Any retract of an object in $\C$ is in $\C$.
\end{itemize}
\end{paragr}

\begin{prop}\label{prop:saturatedclass}
If $A$ is an elegant Reedy category, any class $\C$ of presheaves on $A$ which
is saturated by monomorphisms and contains
the representable presheaves contains all the presheaves on $A$.
\end{prop}

\begin{proof}
As the boundary inclusions $\bord a\to a$ form a generating family for the
class of monomorphisms, it is sufficient to prove that the
boundaries $\bord a$ belong to $\C$ for any representable presheaf $a$.
We proceed by induction on the dimension $d$ of $a$. If $d\leq 0$, then the only
proper subobject of $a$ is the empty presheaf, which belongs to $\C$ by definition.
If $d>0$, consider the set $E$ of proper subobjects $K$ of $a$ which are in $\C$.
It is clear that $E$ is non-empty because the empty subobject of $a$ is an element of $E$.
Note that a subobject $K$ of $a$ is proper if and only if the identity of $a$ is not
contained in the set of non-degenerate sections of $K$. Therefore, proper subobjects
are stable by arbitrary unions in $a$.
Since any totally ordered set has a cofinal well ordered subset, by
Zorn's lemma, there exists a maximal element $K$ in $E$.
Let us prove that $K=\bord a$. If not, then let us choose a section $u:b\to \bord a$ of $\bord a$
which does not belong to $K$. We may choose $u$ such that the dimension of $b$ is minimal with
respect to this property. Thus $u$ must be non-degenerate, and the boundary $\bord u:\bord b\to \bord a$
must factor through $K$. We then have a pushout square of the following form, where $L\subset \bord a$
denotes the union of $K$ and of the image of $u$ in $\bord a$.
$$\xymatrix{
\bord b\ar[d]\ar[r]^{\bord u}&K\ar[d]\\
b\ar[r]^u&L
}$$
By induction, $\bord b$ is in $\C$, and as $b$ is representable, it belongs to $\C$.
As $K$ is in $\C$ as well, $L$ must be in $\C$, which gives a contradiction.
Therefore, the boundary $\bord a=K$ is in $\C$.
\end{proof}

\begin{prop}\label{prop:elegantregular}
If $A$ is an elegant Reedy category, then
any $A$-localizer is regular. In other words, for any
model category structure on $\pref{A}$ whose cofibrations precisely
are the monomorphisms, and for any presheaf $X$ on $A$, the
family of sections of $X$ exhibit $X$ as an
homotopy colimit of representable presheaves (see \cite[D\'efinition 3.4.13]{Ci3}).
\end{prop}

\begin{proof}
This follows right away from \cite[Exemple 3.4.10, Proposition 3.4.22]{Ci3}
and from the preceding proposition.
\end{proof}

A consequence of the preceding proposition is that, in the category
of presheaves on an elegant Reedy category, the notion of weak
equivalence is local in the following sense.

\begin{cor}\label{cor:localweakequiv}
Let $A$ be an elegant Reedy category, and assume that $\pref{A}$ is endowed with
a model category structure whose cofibrations precisely are the monomorphisms.
Consider a commutative triangle of the form
$$\xymatrix{
X\ar[rr]^f\ar[dr]_p&&Y\ar[dl]^q\\
&S&}$$
in which both $p$ and $q$ are fibrations. Then $f$ is a weak equivalence
if and only if, for any representable presheaf $a$ and any section $s:a\to S$,
the induced morphism $a\times_S X\to a\times_S Y$ is a weak equivalence.
\end{cor}

\begin{proof}
This obviously is a necessary condition (because, the pullback functor along
$a\to S$ is a right Quillen functor from $\pref A /S$ to $\pref A /a$,
and thus, by Ken Brown's lemma,
preserves weak equivalences between fibrant objects).
The converse follows from the preceding proposition and from \cite[Corollaire 3.4.47]{Ci3}.
One can also give a more elementary proof using directly Proposition \ref{prop:saturatedclass}
as follows. Replacing $A$ by $A/S$, we may assume without loss of generality that
$S$ is the terminal object. The class of presheaves $Z$ such that
the induced map $Z\times X\to Z\times Y$ is a weak equivalence is
saturated and contains the representable presheaves, so that it contains
the terminal object by Proposition \ref{prop:saturatedclass}.
\end{proof}

\begin{paragr}\label{par:defEq}
Let $A$ be an elegant Reedy category, and assume that $\pref{A}$ is endowed with
a proper model category structure whose cofibrations precisely are the monomorphisms.
Given two fibrations $p:X\to S$ and $q:Y\to S$, we will write
\begin{equation}\label{eq:relinthom}
\sHom_S(X,Y)\to S
\end{equation}
for the map corresponding to the internal Hom of $\pref{A/S}$
through the equivalence $\pref{A}/S\simeq\pref{A/S}$.
In other words, given a map $T\to S$, morphisms from $T$ to
$\sHom_S(X,Y)$ over $S$ correspond bijectively to morphisms of the form
$T\times_SX\to T\times_SY$ over $T$.
Given any map $T\to S$, we have the canonical pullback square below.
\begin{equation}\label{eq:relinthomcart}
\begin{split}
\xymatrix{
\sHom_T(T\times_SX,T\times_SY)\ar[r]\ar[d]&\sHom_S(X,Y)\ar[d]\\
T\ar[r]&S}\end{split}
\end{equation}
Remark that, if $p$ and $q$ are fibrations, then the map \eqref{eq:relinthom}
is a fibration as well: as our model category structure is right proper,
the pullback functor along $p$ is a left Quillen functor from
$\pref A /S$ to $\pref A /S$, so that its right adjoint is a right Quillen
functor, hence preserves fibrant objects.
In this case, we define a subpresheaf
$\Eq_S(X,Y)\subset \sHom_S(X,Y)$
by requiring that, for any map $T\to S$, a section of $\sHom_S(X,Y)$ over $T$
factors through $\Eq_S(X,Y)$ if and only if the corresponding map
$T\times_SX\to T\times_SY$ is a weak equivalence. This actually defines
a subpresheaf of $\sHom_S(X,Y)$ over $S$ precisely because of
Corollary \ref{cor:localweakequiv}.
\end{paragr}

\begin{prop}\label{prop:Eqfib}
Under the assumptions of paragraph \ref{par:defEq},
for any fibrations $X\to S$ and $Y\to S$, the structural map
$\Eq_S(X,Y)\to S$ is a fibration.
\end{prop}

\begin{proof}
Consider a commutative square of the following form
$$\xymatrix{
K\ar[r]^(.3)k\ar[d]_j&\Eq_S(X,Y)\ar[d]\\
L\ar[r]&S}$$
in which the map $j$ is a trivial cofibration.
As the structural map $\sHom_S(X,Y)\to S$ is a fibration, one can find a lift $l$
in the solid commutative square below.
$$\xymatrix{
K\ar[r]^(.3)k\ar[d]_j&\sHom_S(X,Y)\ar[d]\\
L\ar[r]\ar@{..>}[ur]^l&S}$$
The equation $lj=k$ means that we have a pullback square of the form
$$\xymatrix{
K\times_SX\ar[r]^k\ar[d]_{j\times_S 1_X}&K\times_SY\ar[d]^{j\times_S1_X}\\
L\times_SX\ar[r]^l&L\times_SY
}$$
in which the two vertical maps are weak equivalence (by right properness)
as well as $k$. Therefore, the map $l:L\to \sHom_S(X,Y)$
factors through $\Eq_S(X,Y)$, which produces a lift in the commutative square we started from.
\end{proof}

\begin{definition}
Let $A$ be a small category. A \emph{strongly proper}
model category structure on $\pref{A}$ is a proper model category structure
whose cofibrations precisely are the monomorphisms, and  such that the notion of fibration
is local over $A$ in the following sense:
any morphism of presheaves $p:X\to Y$, the following conditions are equivalent.
\begin{itemize}
\item[(i)] The map $p:X\to Y$ is a fibration.
\item[(ii)] For any representable presheaf $a$ and any section $a\to Y$, the
first projection
$a\times_Y X\to a$ is a fibration.
\end{itemize}
\end{definition}

\begin{paragr}\label{paragr:defstrongproper}
Let $A$ be a small category, and assume that $\pref{A}$ is endowed with
a strongly proper model category structure.
Consider an infinite regular cardinal $\kappa$ which is greater than the cardinal
of the set of arrows of $A$. We will use the construction of Hofmann and Streicher \cite{Str}
of a universe of fibrations with $\kappa$-small fibers.
We denote by $\mathit{Set}_\kappa$
some full subcategory of the category of sets of cardinal lesser than $\kappa$,
such that, for any cardinal $\alpha<\kappa$, there exists a
set of cardinal $\alpha$ in $\mathit{Set}_\kappa$.
Let $W$ be the presheaf whose set of sections over an object $a$ of $A$
is the set of functors $\op{(A/a)}\to \mathit{Set}_\kappa$.
Given a map $f:a\to b$ in $A$, the precomposition with the
induced functor $A/a\to A/b$ defines the corresponding map $f^*:W_b\to W_a$.
Similarly, let $\overline{W}$ be the presheaf whose set of sections
consists of couples $(X,s)$, where $X$ is
a presheaf on $A/a$ with values in $\mathit{Set}_\kappa$ (i.e. an element of $W_a$),
and $s$ is a global section of $X$. Forgetting the sections defines a mophism
of presheaves $\rho:\overline{W}\to W$. Note that, since we have
canonical equivalences $\pref{A}/a\simeq\pref{A/a}$, any element $X$ of $W_a$
determines canonically a morphism $p_X:X\to a$; in fact, one can identify
the elements of $W_a$ as the data of a presheaf $X$ on $A$ with values in $\mathit{Set}_\kappa$,
together with a map $X\to a$, as well as with specified cartesian squares
$$\xymatrix{
f^*(X)\ar[r]\ar[d]&X\ar[d]\\
b\ar[r]^f&a
}$$
for any morphism $f:b\to a$ in $A$ (with $f^*(X)$ a presheaf with
values in $\mathit{Set}_\kappa$).
We define the presheaf $U$ as the subpresheaf of $W$ whose sections over
a representable presheaf $a$ are the elements $X$ such that the corresponding
morphism $p_X:X\to a$ is a fibration. We define $\overline{U}$ by the
following pullback square.
$$\xymatrix{
\overline{U}\ar[d]_\pi\ar[r]&\overline{W}\ar[d]^\rho\\
U\ar[r]&W
}$$
The following lemma is straightforward.
\end{paragr}

\begin{lemma}
Assume that there are two cartesian squares of the following
form.
$$\xymatrix{
X\ar[r]^u\ar[d]_p&X'\ar[d]^{p'}&X\ar[r]^x\ar[d]_p&\overline{W}\ar[d]^\rho\\
Y\ar[r]^v&Y'&Y\ar[r]^y&W
}$$
If $p'$ has $\kappa$-small fibers and $v$ is a monomorphism, then
there exists a map $y':Y'\to W$ such that $y'v=y$. In particular, the case
where $Y$ is empty tells us that a morphism of presheaves over $A$ has
$\kappa$-small fibers if and only if it can be obtained as a pullback
of the morphism $\rho:\overline{W}\to W$.
\end{lemma}

\begin{prop}\label{prop:classfib1}
Under the assumptions of paragraph \ref{paragr:defstrongproper},
let $p:X\to Y$ be map with $\kappa$-small fibers, and choose
a classifying cartesian square.
$$\xymatrix{
X\ar[r]\ar[d]_p&\overline{W}\ar[d]^\rho\\
Y\ar[r]^y&W
}$$
Then $p$ is a fibration if and only if the classifying map $y$
factors through $U\subset W$.
\end{prop}

\begin{proof}
This is a reformulation of the last part of the definition of strong properness.
\end{proof}

\begin{cor}\label{cor:classfib2}
Assume that there are two cartesian squares of the following
form.
$$\xymatrix{
X\ar[r]^u\ar[d]_p&X'\ar[d]^{p'}&X\ar[r]^x\ar[d]_p&\overline{U}\ar[d]^\pi\\
Y\ar[r]^v&Y'&Y\ar[r]^y&U
}$$
If $p'$ is a fibration with $\kappa$-small fibers and if $v$ is a monomorphism, then
there exists a map $y':Y'\to U$ such that $y'v=y$.
\end{cor}

\begin{thm}\label{thm:universekappa2}
Under the assumptions of paragraph \ref{paragr:defstrongproper}, if $A$
is an Eilenberg-Zilber category and if $\kappa$ is an inaccessible cardinal,
the map $\pi:\overline{U}\to U$ is a univalent fibration between fibrant objects
which classifies fibrations with $\kappa$-small fibers.
\end{thm}

\begin{proof}
The fact that $U$ is fibrant is a reformulation of proposition \ref{prop:extensionfib}
and of Corollary \ref{cor:classfib2}. The fact that $\pi$ is a fibration
follows straight away from Proposition \ref{prop:classfib1}.
Let $\pi_0:\overline{U}_0=\overline{U}\times U\to U\times U$ and
$\pi_1:\overline{U}_1=U\times\overline{U}\to U\times U$ be the pullbacks
of the fibration $\pi$ along the first and second projection of $U\times U$ to $U$,
respectively. By virtue of Proposition \ref{prop:Eqfib}, we have a canonical
fibration
$$(s,t):\Eq_{U\times U}(\overline{U}_0,\overline{U}_1)\to U\times U\, .$$
As the pullback of the fibration $\pi_i$ along the diagonal $U\to U\times U$
is canonically isomorphic to $\pi$ for $i=0,1$, the fibration $(s,t)$
has a canonical section over the diagonal $U\to U\times U$, which provides
a morphism
$$\mathit{id}:U\to \Eq_{U\times U}(\overline{U}_0,\overline{U}_1)$$
such that $(s,t)\mathit{id}$ is the diagonal (or equivalently, such that
$s\mathit{id}=t\mathit{id}=1_U$).
The property that $\pi$ is univalent means that this map $\mathit{id}$
is a weak equivalence. It is thus sufficient to prove that the fibration
$$t:\Eq_{U\times U}(\overline{U}_0,\overline{U}_1)\to U$$
is a trivial fibration.
Consider a cofibration $j:Y\to Y'$. Then a commutative square
$$\xymatrix{
Y\ar[r]^(.3)\xi\ar[d]_j&\Eq_{U\times U}(\overline{U}_0,\overline{U}_1)\ar[d]^t\\
Y'\ar[r]^{\xi'}&U
}$$
consists essentially of a commutative diagram of the form
$$\xymatrix{
X_0\ar[r]^w\ar[dr]_{p_0}&X_1\ar[r]^{i_1}\ar[d]^{p_1}&X'_1\ar[d]^{p'_1}\\
&Y\ar[r]^j&Y'
}$$
in which $p_0$, $p_1$ and $p'_1$ are fibrations (with $\kappa$-small fibers), $w$ is a weak
equivalence, and the square is cartesian (where the triple $(p_0,w,p_1)$ correponds to $\xi$, the fibration
$p'_1$ corresponds to $\xi'$, and the cartesian square to the
equation $\xi'j=t\xi$). Therefore, Proposition \ref{prop:extensionweakequivcof}
together with Corollary \ref{cor:classfib2} give a map $\zeta:Y'\to \Eq_{U\times U}(\overline{U}_0,\overline{U}_1)$
such that $t\zeta=\xi'$ and $\zeta j=\xi$.
\end{proof}

\begin{prop}\label{prop:stronglyproperinjectivemod}
Let $A$ and $B$ be a small categories, with $B$ having the structure
of an elegant Reedy category.
Assume that $\pref{A}$ is endowed with a strongly proper model
category structure, and consider the associated injective
model category structure on the category of presheaves on $B$ with values in $\pref{A}$ (for which
cofibrations and weak equivalences are defined termwise with respect to evaluation
at objects of $B$). Then this defines a strongly proper model category structure
on $\pref{A\times B}$. In particular, this model category structure
supports a model of intensional type theory with dependent sums and products,
identity types, and as many univalent universes as there are inaccessible cardinals
greater than the set of arrows of $A\times B$.
\end{prop}

\begin{proof}
Given a presheaf $X$ on $C=A\times B$ and a presheaf $F$ on $B$, we obtain
a presheaf $\Hom_{\pref B}(F,X)$ on $A$ whose sections over an object $a$
are given by
$$\Hom_{\pref B}(F,X)_a=\Hom_{\pref B}(F,X_a)$$
where $X_a$ is the presheaf on $B$ obtained by evaluating $X$ at $a$.
Given an object $b$ of $B$, the Yoneda lemma for presheaves over $B$
gives the identification
$$X_b=\Hom_{\pref B}(b,X)\, ,$$
and we set
$$M_bX=\Hom_{\pref B}(\bord b,X)\, .$$
The injective model category structure on the category of presheaves on $B$
with values in $\pref A$ coincides with the Reedy model structure.
This means that a morphism of presheaves $p:X\to Y$ on $C$ is
a fibration if and only if, for any object $b$ of $B$, the induced map
$$q_b: X_b\to Y_b\times_{M_bY}M_bX$$
is a fibration of $\pref{A}$.
Consider a map $p:X\to Y$ such that, for any representable presheaf $c$ on $C$
and any section $c\to Y$, the canonical map $c\times_Y X\to c$ is a fibration.
Let $b$ be an object of $B$. We want to prove that $q_b$ is a fibration of $\pref{A}$.
But the model category structure on $\pref{A}$ being strongly proper, it
sufficient to prove the existence of lifts in commutative diagrams of the from
$$\xymatrix{
K\ar[rr]\ar[d]_j&&X_b\ar[d]^{q_b}\\
L\ar@{..>}[urr]\ar[r]&a\ar[r]&Y_b\times_{M_bY}M_bX
}$$
in which $j$ is a trivial cofibration and $a$ is a representable presheaf on $A$.
Such a lifting problem is equivalent to a lifting problem of the form
$$\xymatrix{
L\boxtimes\bord b\cup K\boxtimes b\ar[d]\ar[rr]&&X\ar[d]^p\\
L\boxtimes b\ar@{..>}[urr]\ar[r]&a\boxtimes b\ar[r]&Y}$$
in which, for any presheaves $E$ and $F$ on $A$ and $B$ respectively,
we write $ E\boxtimes F$ for the cartesian product of the pullbacks
of $E$ and $F$ along the projections $A\times B\to A$ and $A\times B\to B$
respectively. But $c=a\boxtimes b$ is then a representable presheaf on $C$,
and we are reduced to a lifting problem of the following form.
$$\xymatrix{
L\boxtimes\bord b\cup K\boxtimes b\ar[d]\ar[r]&c\times_YX\ar[d]\\
L\boxtimes b\ar@{..>}[ur]\ar[r]&c}$$
As the projection $c\times_Y X\to c$ is a fibration, this achieves the proof.
\end{proof}

Examples of Eilenberg-Zilber test categories are
the simplicial category $\cats$, the cubical category
$\Cube$ \eqref{example:cube},
the cubical category with connections $\Cube^c$ \eqref{example:cubical},
and Joyal's categories $\Theta_n$ for $1\leq n\leq \omega$ \eqref{example:joyal}.
Once we are here, we can get a much more explicit proof
of Corollary \ref{cor:elegantunivalence}, at least in the case of
Eilenberg-Zilber local test categories: we apply
Theorem \ref{thm:universekappa2} to the Grothendieck model structure
on the category of presheaves on an Eilenberg-Zilber
local test category (see \ref{par:test}),
which is meaningful for we have the following result.

\begin{thm}\label{thm:testfibrationlocal}
Let $A$ be an elegant local test category.
The Grothendieck model category structure on $\pref{A}$
is strongly proper.
\end{thm}

\begin{proof}
We already know that the Grothendieck model category structure is proper
(this does not use the property that $A$ is an elegant Reedy category and is
true for any local test category). Let $p:X\to Y$ be morphism such that,
for any section $a\to Y$, the induced map $a\times_Y X\to a$ is a fibration.
We want to prove that $p$ is a fibration.
Note that $A/Y$ is a again an elegant local test category and that,
under the identification $\pref{A}/Y\simeq\pref{A/Y}$, the
Grothendieck model structure on $\pref{A/Y}$ coincides with the
model category structure on $\pref{A}/Y$ induced by the Grothendieck model
structure on $\pref{A}$.
Therefore, replacing $A$ by $A/Y$, we may assume that $Y$ is the terminal object.
We thus have a presheaf $X$ on $A$ such that $a\times X\to a$ is a fibration
for any representable presheaf $a$, and we want to prove that $X$ is fibrant.
We will consider the minimal model structure on $\pref{A}$ (corresponding
to the minimal $A$-localizer; see \cite[Th\'eor\`eme 1.4.3]{Ci3}), and will prove first
that $X$ is fibrant for the minimal model structure.
Let us choose an interval $I$ such that the projection $Z\times I\to Z$ belongs to the
minimal $A$-localizer (e.g. $I$ might be the subobject classifier; see \cite[1.3.9]{Ci3}).
By virtue of \cite[Remarque 1.3.15, Proposition 1.3.36]{Ci3}, we have to check that the map from $X$ to the
terminal presheaf has the right lifting property with respect to the inclusions
of the form $I\times\bord a\cup\{\varepsilon\}\times a\to I\times a$ for any
representable presheaf $a$ and $\varepsilon=0,1$.
But lifting problems of shape
$$\xymatrix{
I\times\bord a\cup\{\varepsilon\}\times a\ar[r]^(.7)u\ar[d]&X\\
I\times a\ar@{..>}[ur]^v&
}$$
are in bijection with lifting problems of the form
$$\xymatrix{
I\times\bord a\cup\{\varepsilon\}\times a\ar[r]^(.6){(p,u)}\ar[d]&a\times X\ar[d]\\
I\times a\ar@{..>}[ur]^(.45){(pr_2,v)}\ar[r]^(.6){pr_2}&a
}$$
where $p:I\times\bord a\cup\{\varepsilon\}\times a\to a$
is the restriction of the second projection $I\times a\to a$. Hence $X$
is fibrant for the minimal model structure (because the projection
being a fibration for the Grothendieck model structure,
it is also a fibration for the minimal model structure).
By virtue of Proposition \ref{prop:elegantregular} and \cite[Proposition 6.4.26]{Ci3},
the Grothendieck model category structure is the left Bousfield localization
of the minimal model category structure on $\pref{A}$ by the set of maps
between representable presheaves. It is thus sufficient to prove that, for any
map between representable presheaves $u:a\to b$, the map
$$\mathrm{Map}(b,X)\to\mathrm{Map}(a,X)$$
is a weak equivalence (where the mapping spaces are constructed
from the minimal model structure). The latter is equivalent to the map
$$\mathrm{Map}_b(b,b\times X)\to\mathrm{Map}_b(a,b\times X)$$
where $\mathrm{Map}_b$ denotes the mapping space with respect to the
model category structure on $\pref{A}/b$ induced by the minimal model structure
on $\pref{A}$. The projection from $b\times X$ to $b$ being a fibration
of the Grothendieck model category structure, we deduce that
$X$ is local with respect to the left Bousfield localization by the
maps between representable presheaves, and thus that $X$ is fibrant
in the Grothendieck model category structure.
\end{proof}

\begin{rem}
The preceding theorem, together with Proposition \ref{prop:stronglyproperinjectivemod},
gives a new proof, in the case of Eilenberg-Zilber categories,
of Shulman's result that the injective model structure
for simplicial presheaves supports
a model of intensional type theory with univalent universes \cite[Theorem 5.6]{Sh}.
\end{rem}

\begin{rem}
The proof of Theorem \ref{thm:testfibrationlocal} would have been much easier if we would
have exhibited a generating set of trivial cofibrations of the Grothendieck
model category structure of the form $K\to a$ with $a$ representable.
This happens in practice (e.g. horn inclusions for simplicial sets, open boxes
for cubical sets), but I don't know if this is true for a general
elegant local test category. In fact, there is a candidate for a counter-example.
Let $\Omega$ be the category of finite rooted trees considered by Weiss
and Moerdijk for their notion of dendroidal sets. In a short note in preparation
(in collaboration with D.~Ara and I.~Moerdijk), it will be shown that $\Omega$
is a test category. Although $\Omega$ is not an elegant Reedy category for the
simple reason that it is not a Reedy category, for any normal dendroidal set $X$
(i.e. such that, for any tree $T$, the automorphisms of $T$ act freely on the
set of sections of $X$ over $T$), the category $\Omega/X$ is an Eilenberg-Zilber
Reedy category. As a consequence, given any reasonable model of the operad $E_\infty$
(that is any weakly contractible normal dendroidal set), the category $\Omega/E_\infty$
is an Eilenberg-Zilber test category. On the other hand the dendroidal horns are well understood:
the right lifting property with respect to dendroidal
inner horns define the $\infty$-operads, which are models for topological (coloured)
symmetric operads. But if we consider the right lifting property
with respect to all dendroidal horns, Ba\v{s}i\'c and Nikolaus~\cite{BaNa}
have shown that we obtain models of infinite loop spaces.
This means that we have an Eilenberg-Zilber test category $\Omega/E_\infty$
for which there really is no natural candidate for a generating family of
trivial cofibrations with representable codomains.
\end{rem}
%

\nocite{JaCub,GZ}
\bibliography{HoTT}

\providecommand{\bysame}{\leavevmode ---\ }
\providecommand{\og}{``}
\providecommand{\fg}{''}
\providecommand{\smfandname}{\&}
\providecommand{\smfedsname}{\'eds.}
\providecommand{\smfedname}{\'ed.}
\providecommand{\smfmastersthesisname}{M\'emoire}
\providecommand{\smfphdthesisname}{Th\`ese}
\begin{thebibliography}{BCH14}

\bibitem[BCH14]{BCH}
{\scshape M.~Bezem, T.~Coquand {\normalfont \smfandname} S.~Huber} -- {\og A
  model of type theory in cubical sets\fg}, preprint, 2014.

\bibitem[BN12]{BaNa}
{\scshape M.~Ba\v{s}i\'c {\normalfont \smfandname} T.~Nikolaus} -- {\og
  Dendroidal sets as models for connective spectra\fg}, arXiv$:$1203.6891,
  2012.

\bibitem[BR13]{BeRe}
{\scshape J.~E. Bergner {\normalfont \smfandname} C.~Rezk} -- {\og Reedy
  categories and the {$\Theta$}-construction\fg}, \emph{Math. Z.} \textbf{274}
  (2013), no.~1, p.~499--514.

\bibitem[Cis06]{Ci3}
{\scshape D.-C. Cisinski} -- \emph{Les pr\'efaisceaux comme mod\`eles des types
  d'homotopie}, Ast\'erisque, vol. 308, Soc. Math. France, 2006.

\bibitem[CM11]{CisMal}
{\scshape D.-C. Cisinski {\normalfont \smfandname} G.~Maltsiniotis} -- {\og La
  cat\'egorie {$\Theta$} de joyal est une cat\'egorie test\fg}, \emph{J. Pure
  Appl. Algebra} \textbf{215} (2011), no.~5, p.~962--982.

\bibitem[GK12]{GeKo}
{\scshape D.~Gepner {\normalfont \smfandname} J.~Koch} -- {\og Univalence in
  locally cartesian closed $\infty$-categories\fg}, arXiv$:$1208.1749v2, 2012.

\bibitem[GZ67]{GZ}
{\scshape P.~Gabriel {\normalfont \smfandname} M.~Zisman} -- \emph{Calculus of
  fractions and homotopy theory}, Ergebnisse der Mathematik, vol.~35,
  Springer-Verlag, 1967.

\bibitem[Jar06]{JaCub}
{\scshape J.~F. Jardine} -- {\og Categorical homotopy theory\fg},
  \emph{Homology Homotopy Appl.} \textbf{8} (2006), no.~1, p.~71--144.

\bibitem[KLV12]{KLV}
{\scshape C.~Kapulkin, P.~L. Lumsdaine {\normalfont \smfandname} V.~Voevodsky}
  -- {\og Univalence in simplicial sets\fg}, arXiv$:$1203.2553v3, 2012.

\bibitem[Lur09]{Lurie}
{\scshape J.~Lurie} -- \emph{Higher topos theory}, Annals of Mathematics
  Studies, vol. 170, Princeton University Press, 2009.

\bibitem[LW]{LuWa}
{\scshape P.~L. Lumsdaine {\normalfont \smfandname} M.~Warren} -- {\og The
  local universe model of type theory\fg}, in preparation.

\bibitem[Mal05]{Ma}
{\scshape G.~Maltsiniotis} -- \emph{La th\'eorie de l'homotopie de
  {G}rothendieck}, Ast\'erisque, vol. 301, Soc. Math. France, 2005.

\bibitem[Mal09]{Ma7}
\bysame , {\og La cat\'egorie cubique avec connexions est une cat\'egorie test
  stricte\fg}, \emph{Homology Homotopy Appl.} \textbf{11} (2009), no.~2,
  p.~309--982.

\bibitem[Shu12]{ShInv}
{\scshape M.~Shulman} -- {\og The univalence axiom for inverse diagrams\fg},
  arXiv$:$1203.3253, 2012.

\bibitem[Shu13]{Sh}
\bysame , {\og The univalence axiom for elegant {R}eedy presheaves\fg},
  arXiv$:$1307.6248, 2013.

\bibitem[Str14]{Str}
{\scshape T.~Streicher} -- {\og A model of type theory in simplicial sets. {A}
  brief introduction to {V}oevodsky's homotopy type theory\fg}, \emph{J. Appl.
  Log.} \textbf{12} (2014), p.~45--49.

\end{thebibliography}
\bibliographystyle{smfalpha}
\end{document}